\documentclass[leqno,11pt,letterpaper]{smfart}
\usepackage{amsmath,amssymb,amsthm,graphicx,url}
\usepackage{epstopdf}
\usepackage{color}

\newcommand{\xqedhere}[1]{%
    \rlap{%
         \hbox to#1{%
           \hfil
           \llap{%
               \ensuremath{\square}
           }%
       }%
   }%
}

\def\pasdegrille{\let\grille = \pasgrille}
\def\ecriture#1#2{\setbox1=\hbox{#1} 
\dimen1= \wd1 \dimen2=\ht1 \dimen3=\dp1 \grille #2 \box1 }
\def\aat#1#2#3{
\divide \dimen1 by 48 \dimen3=\dimen1 \multiply \dimen1 by #1
\advance \dimen1 by -\dimen3 \divide \dimen1 by 101 \multiply
\dimen1 by 100 \divide \dimen2 by \count11 \multiply \dimen2 by #2
\setbox0=\hbox{#3}\ht0=0pt\dp0=0pt
  \rlap{\kern\dimen1 \vbox to0pt{\kern-\dimen2\box0\vss}}\dimen1= \wd1
\dimen2=\ht1}
\def\pasgrille{
\count12= \dimen1 \divide \count12 by 50 \divide \dimen2 by
\count12 \count11 =\dimen2 \ \divide \dimen1 by 48
\setlength{\unitlength}{\dimen1} \smash{\rlap{\ }} \dimen1= \wd1
\dimen2=\ht1 }
\def\grille{
\count12= \dimen1 \divide \count12 by 50 \divide \dimen2 by
\count12 \count11 =\dimen2 \ \divide \dimen1 by 48
\setlength{\unitlength}{\dimen1}
\smash{\rlap{\graphpaper[1](0,0)(50, \count11)}} \dimen1= \wd1
\dimen2=\ht1 }

\pasdegrille
\newcommand{\be}{\begin{equation}}
\newcommand{\ee}{\end{equation}}

\newcommand{\NN}{{\mathbb N}}

\newcommand{\RR}{{\mathbb R}}

\renewcommand{\Re}{\mathop{\rm Re}\nolimits}
\renewcommand{\Im}{\mathop{\rm Im}\nolimits}

\theoremstyle{plain}

\newtheorem{thm}{Theorem}

\newtheorem{prop}{Proposition}[section]
\newtheorem{cor}[prop]{Corollary}
\newtheorem{lem}[prop]{Lemma}

\theoremstyle{definition}

\newtheorem{rem}[prop]{Remark}
\newtheorem{defn}[prop]{Definition} 

\numberwithin{equation}{section}

\def\squarebox#1{\hbox to #1{\hfill\vbox to #1{\vfill}}}

\usepackage{amsxtra}

\ifx\pdfoutput\undefined
  \DeclareGraphicsExtensions{.pstex, .eps}
\else
  \ifx\pdfoutput\relax
    \DeclareGraphicsExtensions{.pstex, .eps}
  \else
    \ifnum\pdfoutput>0
      \DeclareGraphicsExtensions{.pdf}
    \else
      \DeclareGraphicsExtensions{.pstex, .eps}
    \fi
  \fi
\fi

\title[Semi-classical measures]{Semi-classical measures for inhomogeneous Schr\"odinger equations on tori}

\author[N. Burq]{Nicolas Burq}
\address{Universit{\'e} Paris Sud, Math{\'e}matiques, B{\^a}t 425, 91405  Orsay Cedex, France,  UMR 8628 du CNRS and Ecole Normale Sup\'erieure, 45, rue d'Ulm, 75005 Paris,  Cedex 05,  France, UMR 8553 du CNRS}
\email{Nicolas.burq@math.u-psud.fr}
\usepackage{amssymb}
\usepackage{amsmath, amsthm, amsopn, amsfonts}

\def\11{{\rm 1~\hspace{-1.4ex}l} }
\def\R{\mathbb R}
\def\C{\mathbb C}
\def\Z{\mathbb Z}
\def\N{\mathbb N}

\def\T{\mathbb T}
\def\O{\mathbb O}

\begin{document}

\begin{abstract}
The purpose of this note is to investigate the high frequency  behaviour of solutions to linear Schr\"odinger equations. More precisely,  Bourgain~\cite{Bo} and Anantharaman-Macia~\cite{AM} proved that any weak-$*$ limit of the square density of solutions to the time dependent homogeneous Schr\"odinger equation is absolutely continuous with respect to the Lebesgue measure on $\mathbb{R} \times \T^d$. Our contribution
% is to show that on the one hand, following a strategy iniciated in~\cite{BZ2}, this result can be precised in a quantitative way (at least on the two dimensional case), and on the other hand, 
 is that the same result automatically holds for non homogeneous Schr\"odinger equations, which allows for abstract potential type perturbations of the Laplace operator.
%In the dynamical case $ V $ can be made dependendent on
%time.
\end{abstract}   

\maketitle   

\section{Introduction}  We are interested in this note in understanding the high frequency behaviour of solutions of  linear Schr\"odinger equations on tori, $\T^d= \R^d / \Z^d$. Consider a sequence of initial data $(u_{0,n})$, bounded in $L^2( \T^d)$ and denote by $(u_n)$ the sequence of solutions to Schr\"odinger equation and $(\nu_n)$ their concentration measures given by 
$$ u_n = e^{it\Delta} u_{0,n}, \qquad \nu_n = |u_n|^2(t,x) dtdx,$$
The sequence $\nu_n$ on $\RR_t \times \T^d$ is bounded (in mass) on any time interval $(0,T)$ by \smash{$T \sup_n \|u_{0,n}\|_{L^2( \T^d)}^2$}. The following result was proved by Bourgain~\cite[Remark p 108]{Bo} and later by Anantharaman-Macia~\cite[Theorem 1]{AM} by a completely different approach, following a more geometric path (see also~\cite{Ja, Mac, BZ2, BZ3, AJM} for related works).
\begin{thm}\label{theorem.1}
Any weak-$*$ limit of the sequence $(\nu_n)$ is absolutely continuous with respect to the Lebesgue measure $dtdx$ on $\RR_t \times \T^d$.
\end{thm}
\begin{rem} Actually, in~\cite{AM} a more precise description of the possible limits is given and the result is proved in the case of Schr\"odinger operators $\Delta + V(t,x)$, if $V\in L^\infty ( \RR_t \times \T^2)$ is also continuous except possibly on a set of (space-time) Lebesgue measure $0$.
 \end{rem}
 The purpose of this note is to show that the result in Theorem~\ref{theorem.1} extends
  to the case of solutions to the non-homogeneous Schr\"odinger equation, and consequently to the case of Schr\"odinger operators $\Delta+ V$ where $V\in L^1_{\text{loc}}(\R_t; \mathcal{L} (L^2( \T^d)))$ (we also give as  an illustration an application to a simple non linear equation). Let us emphasize that our approach uses no particular property of the Laplace operator on tori other than self-adjointness (to get $L^2$ bounds for the time evolution) and the fact that Theorem~\ref{theorem.1} holds, which is used as a black box, and establishes an abstract link between the study of  weak-$*$ limits of solutions of the homogeneous and inhomogeneous Schr\"odinger equations. 

{\bf Acknowledgements.} I would like to thank P. G\'erard for suggesting the application in Section~\ref{sec.3}. I also acknowledge partial support from the Agence Nationale de la Recherche, project NOSEVOL, 2011 BS01019 01.
\section{Inhomogeneous Schr\"odinger equations}

\begin{defn} Let $T>0$.
For any sequence $(u_n)$ bounded in $L^2((0,T)\times \T^d)$, we say that the sequence $(u_n)$ satisfies property ($AC_T$) if any weak-$*$ limit, $\nu$ of $(\nu_n)$ is absolutely continuous with respect to the Lebesgue measure on $(0,T)\times \T^d$.
\end{defn}
\begin{thm}\label{theorem.2}
Let $(u_{n,0})$ and $(f_{n})$ be two sequences bounded in $L^2( \T^d)$ and $L^1_{\text{loc}}(\RR_t; L^2( \T^d))$ respectively. Let $u_n$ be the solution of 
$$ (i \partial_t + \Delta ) u_n=f_n, \qquad u_n \mid_{t=0} = u_{n,0}, \qquad u_n = e^{it \Delta} u_{n,0}+ \frac 1 i \int_0^t e^{i(t-s) \Delta} f_n (s) ds.$$
Then for any $T>0$, the sequence $(u_n)$, which is clearly bounded in $L^2( (0,T) \times \T^2)$ by 
$$T^{1/2} \sup_n\bigl(\|u_{n,0}\|_{L^2(\T^d)} + \|f_n\|_{L^1((0,T); L^2( \T^d))} \bigr), $$
  satisfies property ($AC_T$). 
\end{thm}
\begin{cor}\label{corol}
Let $V\in L^1_{\text{loc}}(\RR_t; \mathcal{L} (L^2( \T^2)))$ (for example $V$ can be chosen to be a potential in $ L^1_{\text{loc}}(\RR_t; L^\infty( \T^2))$ acting by pointwize multiplication). For any sequence $(u_{n,0})_{n \in \NN}$ bounded in $L^2( \T^2)$, let $(u_n)$ be the sequence of the unique solutions in $C^0(\RR; L^2( \T^2))$ of 
$$ (i \partial_t + \Delta + V (t)) u_n=0, \qquad u_n \mid_{t=0} = u_{n,0}.$$
Then the sequence $(u_n)$ satisfies for any $T>0$ the property ($AC_T$). 
\end{cor}
Indeed, since 
$$
 \frac{d} {dt} \|u_n\|_{L^2( \T^d)} ^2 =  2\Re \bigl( \partial_t u , u \bigr) _{L^2 ( \T^d)}
 = 2 \Re \bigl( i \Delta u  + iVu  , u\bigr) _{L^2 ( \T^d)} = -2\Im ( V u, u)_{L^2( \T^d)}
 $$
 we obtain by Gronwall inequality 
 $$ \| u_n(t)\|_{L^2(\T^d)}^2 \leq \|u_{n,0} \|_{L^2( \T^d)}^2 e^{\int_0^t \| V(s) \|_{\mathcal{L} ( L^2 ( \T^d)}ds},$$
 and consequently  the sequence  $(f_n)= (-V (t) u_n)$ is clearly bounded in $L^1_{\text{loc}}(\RR_t; L^2( \T^d))$ and we can apply Theorem~\ref{theorem.2}.
 \begin{rem} Any time independent $V\in \mathcal{L}( L^2( \T^d))$ satisfies the assumptions above, and consequently, if $(u_n)$ is a sequence of $L^2$ normalized eigenfunctions of $\Delta + V$, it follows from Corollary~\ref{corol} that any weak-$*$ limit of $|u_n|^2(x) dx$ is absolutely continuous with respect to the Lebesgue measure on $\T^d$. The proof we present below seems to be intrinsically a time dependent proof. However, it would be interesting to obtain a proof  of this result avoiding the detour via the study of the time dependent Schr\"odinger equation. 
 \end{rem}  
\begin{proof}[Proof of Theorem~\ref{theorem.2}.]
Notice first that if $(u_n)$ satisfies property ($AC_T$), then the sequence $(u_n+v_n)$ satisfies property ($AC_T$) iff the sequence $(v_n)$ satisfies property ($AC_T$), because, if $|u_n|^2 dtdx$ and $|v_n|^2 dtdx$ are converging weakly to $\nu$ and $\mu$ respectively, then according to Cauchy-Schwarz inequality any weak-$*$ limit of $|u_n+v_n|^2 dtdx$ is absolutely continuous with respect to $\nu+ \mu$. The following result shows that the set of sequences satisfying property ($AC_T$) is closed in some weak-strong topology. 
\begin{lem}\label{lemme.3.2}
Consider $(u_n)$ bounded in $L^2((0,T) \times \T^2)$. Assume  that there exists for any $k \in \NN$ a sequence $(u_n^{(k)})_{n\in \mathbb{N}}$ such that 
\begin{enumerate}
\item For any $k$, the sequence  $(u_n^{(k)})_{n\in \mathbb{N}}$ satisfies Property ($AC_T$)
\item The  sequences $(u_n^{(k)})_{n\in \NN}$ are approximating the sequence $(u_n)$ in the following sense. 
\begin{equation}\label{eq.approx}
\lim_{k\rightarrow + \infty} \limsup_{n\rightarrow + \infty} \|u_n - u_{n}^{(k)}\|_{L^2((0,T) \times \T^2)} =0.
\end{equation}
\end{enumerate}
Then the sequence $(u_n)_{n\in \mathbb{N}}$ satisfies property ($AC_T$). 
\end{lem}
\begin{proof} Indeed, for any $\epsilon >0$, let $k_0$ be such that  for any $k\geq k_0$,
$$\limsup_n \|u_n - u_{n,k}\|_{L^2((0,T) \times \T^2)} < \epsilon . $$
Then, if $\nu$ and $\nu^{(k)}$ are weak-$*$ limits of the sequences $(u_n)_{n\in \NN}$ and $(u_n^{(k)})_{n\in \NN}$ respectively, associated to the same subsequence $n_p\rightarrow + \infty$, we have for any  $f\in C^0((0,T)\times \T^2)$ and large~$n$, 
\begin{multline}\label{eq.1}
\int_{(0,T) \times \T^2} |u_{n_p}|^2 \chi dx dt \leq \int_{(0,T) \times \T^2} 2 (|u_{n_p}- u_{{n_p}}^{(k)}|^2+ |u_{{n_p}}^{(k)} |^2 ) dx dt\\
\leq 2 \epsilon^2 + 2 \int_{(0,T) \times \T^2} 2 |u_{{n_p}}^{(k)}| ^2 ) \chi dx dt.
\end{multline}
Passing to the limit $p\rightarrow + \infty $ we obtain
$$ \langle \nu, \chi\rangle \leq 2\epsilon ^2 + 2 \langle \nu^{(k)}, \chi \rangle $$
On the other hand, according to Riesz Theorem (see e.g.~\cite[Theorem 2.14]{Ru}) the measures~$\nu, \nu^{(k)}$ which are defined on the Borelian $\sigma$-algebra, $\mathcal{M}$, are {\em regular} and consequently 
\begin{equation}
\begin{gathered}
 \forall E\in \mathcal{M},
 \nu(E) = \sup_{F closed, F \subset E} \nu(U)= \inf_{U open, E \subset U} \nu(U),\\
  \forall E\in \mathcal{M},  \nu^{(k)}(E) = \sup_{F closed, F \subset E} \nu^{(k)}(U)= \inf_{U open, E \subset U} \nu^{(k)}(U).
  \end{gathered}\end{equation}
  For any $E \in \mathcal{M}$, taking $F_p \subset E $ and $E\subset O_p$ such that
  $$ \lim_{p\rightarrow + \infty} \nu(F_p) = \nu(E), \lim_{p\rightarrow + \infty} \nu^{(k)} (O_p ) = \nu^{(k)} (E)
  $$ and $\chi_p\in C_0( (0,1)\times \T^d; [0,1]) $ equal to $1$ on $F_p$ and supported in $O_p$, we obtain according to~\eqref{eq.1}
  $$ \nu(E) \leq 2 \epsilon ^2 + 2 \nu^{(k)}(E). $$
  Consider now  $E$ a subset of $(0,T) \times \T^d$-Lebesgue measure $0$. Since by assumption $\nu^{(k)}$ is absolutely continuous with respect to the Lebesgue measure, we have $\nu^{(k)}(E) =0$, and hence $\nu(E)\leq 2 \epsilon ^2$ and consequently,  since $\epsilon>0$ can be taken arbitrarily small, we have  $\nu(E)=0$, which proves that $\nu$ is also absolutely continuous with respect to the Lebesgue measure.
  \end{proof} 

We come back to the proof of Theorem~\ref{theorem.2} and fix $T>0$. According to Duhamel formula. 
$$ u_n = e^{it \Delta} u_{0,n} + \frac 1 i \int_0^t  e^{i(t-s)\Delta } f_n (s) ds.$$
According to the remark above, since we know that the sequence $(e^{it \Delta} u_{0,n})$ satisfies property ($AC_T$), it is enough to prove that the sequence $ (v_n)= (\int_0^t  e^{i(t-s)} f_n (s) ds)$ satisfies property ($AC_T$). The key point of the analysis is to remark that  if instead of $v_n$ we had 
$$\tilde{v}_n =\int_0^T  e^{i(t-s)\Delta} V u_n (s) ds= e^{it \Delta} g_n, \qquad g_n = \int_0^T  e^{-is\Delta} V e^{is (\Delta+V)} u_{n,0} (s) ds,
$$
then, we could conclude using Theorem~\ref{theorem.1} because $\tilde{v}_n$ is a solution to the homogeneous Schr\"odinger equation with initial data the bounded sequence $(g_n)$. To pass from $\tilde{v}_n$ to $v_n$, we adapt an idea borrowed from harmonic analysis (Christ-Kiselev' Lemma~\cite{CK}), in the simple form written in Burq-Planchon~\cite{BuPl} (see also~\cite{Bu}). Here the idea is to show that the sequence $(v_n)$ can be approximated by other sequences $(v_n^{(k)})$ in the sense of~\eqref{eq.approx} (actually, we get a stronger convergence, as we can replace the $\limsup$ in~\eqref{eq.approx} by a $\sup$), where each $(v_n^{(k)})$ is a finite sum of solutions of the homogeneous Schr\"odinger equation, properly truncated in time, and hence satisfy property~($AC_T$). Let  
$$ \|f_n \|_{L^1((0,T); L^2( \T^2))} = c_n \leq C.$$
We decompose the interval $(0,T)$ into dyadic pieces on which the $L^1((0,T); L^2( \T^d))$-norm of $f_n$ is equal to $2^{-q}c_n$. For this, we construct recursively (on the index $q\in \NN$) sequences $(t_{p,q,n})_{\genfrac{}{}{0pt}{}{q\in \N \hfill}{p=1,\dots, 2^{q} }}$ such that 
\begin{itemize}
\item $0= t_{0,q,n}< t_{1,q,n}< \cdots < t_{2^{q},q,n}=T$,
\item  $\|f_n \|_{L^1((t_{p,q,n},t_{p+1,q,n}); L^2( \T^2))}= 2^{-q} c_n $,
\item for any $p=0, \cdots 2^{q-1}$, $t_{2p,q, n}= t_{p,q-1,n}$.
\end{itemize}
Notice that if the function 
$$G_n: t\in [0,T] \mapsto \|f_n\|_{L^1((0,t); L^2( \T^d))} \in [0,c_n]$$ is strictly increasing, the points $t_{p,q,n}$ are uniquely determined by the relation $ G_n (t_{p,q,n}) = p2^{-q} c_n$, 
and the last condition above is automatic. In the general case,  the function $G_n$ (which is clearly non decreasing) can have some flat parts, consequently the points $t_{p,q,n}$ may be no more unique and the last condition above ensures that the choice made at step $q+1$ is consistent with the choice made at step $q$.
 For $j=0, \dots, 2^{q} -1$, let   
 $$I_{j,q,n}= [t_{2j,q,n} , t_{2j+1, q,n}[, \qquad J_{j,q,n}= [ t_{2j+1,q,n}, t_{2j+2,q,n}[, \qquad Q_{j,q,n}= J_{j,q,n} \times I_{j,q,n}.$$
 Notice that 
 $$ \{((t,s)\in [0,T[^2; s\leq t \} = \bigsqcup_{q=0}^{+ \infty} \bigsqcup_{j=0}^{2^{q} -1} Q_{j,q,n} \Rightarrow 1_{s\leq t} = \sum_{q=0}^{+ \infty} \sum_{j=0}^{2^{q} -1} 1_{Q_{j,q,n}}(t,s).$$
 \begin{figure}[ht]
\vskip - .5cm $$\ecriture{\includegraphics[width=5cm]{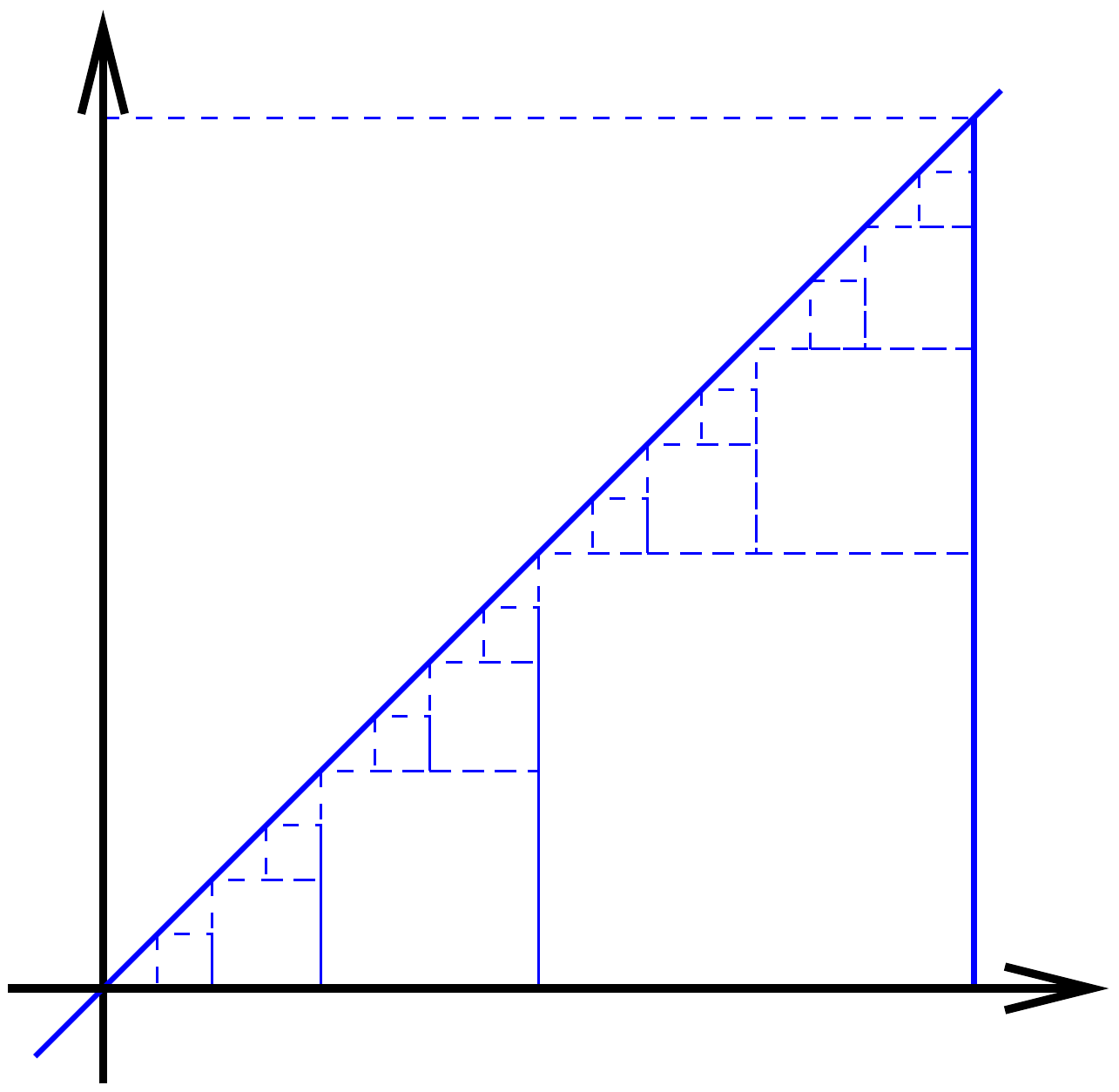}}{
\aat{42}{1}{$T$}\aat{1}{42}{$T$}\aat{32}{13}{${Q}_{0,0}$}\aat{17}{9}{${Q}_{0,1}$}\aat{37}{28}{${Q}_{1,1}$}\aat{10}{6}{${Q}_{0,2}$}\aat{20}{16}{${Q}_{1,2}$}\aat{29}{26}{${Q}_{2,2}$}\aat{39}{35}{${Q}_{3,2}$}
}$$
\vskip -.5cm\caption{Decomposition of a triangle as a union of disjoint squares}
\end{figure}
We now have (if we are able to prove that the series in $q$ converges)
\begin{multline}\label{eq.series}
 v_n = \int_0^t  e^{i(t-s) \Delta} f_n (s) ds = \int_0^T 1_{s\leq t} e^{i(t-s) \Delta} f_n (s) ds \\
= \sum_{q=0}^{+\infty} \sum_{j=0}^{2^{q} -1}  1_{t\in J_{j,q,n}} \int_0^T  e^{i(t-s) \Delta} 1_{s\in I_{j,q,n}}f_n (s) ds 
= \sum_{q=0}^{+\infty} \sum_{j=0}^{2^{q} -1}  1_{t\in J_{j,q,n}} e^{it \Delta} g_{j,q,n} ds,
\end{multline}
with 
\begin{equation}\label{eq.estim}
\begin{gathered}
g_{j,q,n} (x)= \int_0^T  e^{-is \Delta} 1_{s\in I_{j,q,n}}f_n (s)ds = \int_{t_{2j,q,n}} ^{t_{2j+1,q,n}}  e^{-is \Delta} f_n (s)ds, \\
\|g_{j,q,n}\|_{L^2( \T^d)}\leq  \|f_{n}\|_{L^1((t_{2j,q,n}, t_{2j+1, q,n}T); L^2( \T^d))}= 2^{-q} c_n.
\end{gathered}
\end{equation}
Let 
$$v_n^{(k)}= \sum_{q=0}^{k} \sum_{j=0}^{2^{q} -1}  1_{t\in J_{j,q,n}} e^{it \Delta} g_{j,q,n} ds.
$$
Noticing that if a sequence $(w_n)$ satisfies property($AC_T$), then for any sequences $0\leq t_{1,n} <t_{2,n} \leq T$, the sequence $(1_{t\in (t_{1,n}, t_{2,n})}w_n)$ satisfies property($AC_T$), we see that for any $k\in \NN$, the sequence $(v_n^{(k)})$ satisfies property ($AC_T$). On the other hand, since for $j\neq j'$, $1_{t\in J_{j,q,n}}$ and $1_{t\in J_{j',q,n}}$ have disjoint supports, we get, according to~\eqref{eq.estim},
\begin{multline}
\| \sum_{j=0}^{2^{q} -1}  1_{t\in J_{j,q,n}} e^{it \Delta} g_{j,q,n}\|_{L^\infty((0,T); L^2( \T^d))} \leq \sup_{0\leq j\leq 2^{q} -1}   \|1_{t\in J_{j,q,n}} e^{it \Delta} g_{j,q,n}\|_{L^\infty((0,T); L^2( \T^d))}\\
\leq \sup_{0\leq j\leq 2^{q} -1}  \|g_{j,q,n}\|_{L^2( \T^d))} \leq 2^{-q} c_n
\end{multline}
As a consequence, we get that the series~\eqref{eq.series} is convergent and 
$$ \| v_n- v_n^{(k)}\|_{L^2((0,T)\times \T^d)}  \leq \sqrt{T} c_n 2^{-k} \leq C 2^{-k},$$
which, according to Lemma~\ref{lemme.3.2}, concludes the proof of Theorem~\ref{theorem.2}.
\end{proof}
\section{An illustration}\label{sec.3}
We consider here the following non-linear Schr\"odinger equation
\begin{equation}\label{NLS}
(i \partial_t + \Delta)u  + V(u,t) u=0, \text{ on } \mathbb{T}^d, \qquad u \mid_{t=0} =0
\end{equation}
where the the function $z\in \C \mapsto V(z,t) z \in \mathbb{C}$ is globally Lipshitz with respect to the $z$ variable, with time integrable Lipschitz constant: 
$$ \exists C>0; \forall z, z' \in \mathbb{C},  |V(z,t) z -V(z',t) z'| \leq C(t) |z-z'|, C \in L^1_{\text{loc}} ( \mathbb{R}).$$
Notice that for example the choice $V(u,t) = \frac{ |u|^2} {1+ \epsilon |u|^2}$ satisfies these assumptions for any $\epsilon >0$.
\begin{prop}\label{propo.NLS}
For any $u_0 \in L^2( \T^d)$, there exists a unique solution $u \in C(\R; L^2( \T^d))$ to~\eqref{NLS}. Furthermore, there exists a continuous increasing function, $F(t)$ such that for any $u_0\in L^2(\T^d)$, the solution $u$ satisfies 
\begin{equation}
\label{eq.bound}
 \| u\|_{L^2( \T^d)} (t) \leq F(t) \|u_0 \|_{L^2( \T^d)}.
 \end{equation}
\end{prop}
\begin{cor} For any sequence of initial data $(u_{0,n})$ bounded in $L^2( \T^d)$, the sequence  $(u_n)$ of solutions to~\eqref{NLS} satisfies  
$$ \| V(u_n,t) u_n\|_{L^2( \T^d)} \leq C(t) \| u_n\| _{L^{\infty}((0,t); L^2( \T^d))}  \leq C(t)f(t) \|u_{0,n}\|_{L^2( \T^d)}  \in L^1_{\text{ loc}} (\R_t),$$
and consequently, the sequence $(u_n)$ satisfies property $(AC_T)$ for any $T>0$. 
\end{cor}
\begin{proof} [Proof of Proposition~\ref{propo.NLS}] Let 
$$ K: u \in L^\infty((0,T); L^2( \T^d)) \mapsto e^{it\Delta } u_0 + \frac 1 i \int _0 ^t e^{i(t-s)} \bigl(V(u(s),s) u(s)\bigr) ds.
$$ 
We have
\begin{equation}
\begin{aligned}
 \| K(u) - e^{it\Delta } u_0\|_{ L^\infty((0,T); L^2( \T^d))} &\leq \int _0^T C(s) ds \| u\|_{L^\infty((0,T); L^2( \T^d))}\\
 \| K(u) - K(v)\|_{ L^\infty((0,T); L^2( \T^d))} &\leq \int _0^T C(s) ds \|u-v\|_{L^\infty((0,T); L^2( \T^d))}
 \end{aligned}
 \end{equation}
 We obtain that the map $K$ has a unique fixed point on the ball centered on $e^{it \Delta} u_0$ with radius $\|u_0\|_{L^2( \T^d)}$ in $ L^\infty((0,T); L^2( \T^d))$, as soon as $ \int _0 ^T C(s) ds \leq \frac 1 2$.
 This proves the local existence claim. To obtain existence on any time interval $[0,\widetilde{T}]$, we  write $ [0,\widetilde{T}] = \cup_{j=1}^N [t_j, t_{j+1}]$,
 where we choose $t_j$ recursively such that $ \int_{t_j} ^{t_{j+1}} C(s) ds \leq \frac 1 2$.
 Remark that taking $\int_{t_j} ^{t_{j+1}} C(s) ds = \frac 1 2, \forall j< N-1$ gives the bound 
 \begin{equation}\label{borne}
  N \leq  1+ 2 \int_0^{\widetilde{T}} C(s) ds.
  \end{equation}
 Then applying the first step recursively gives a solution on $[0,\widetilde{T}]$ which satisfies according to~\eqref{borne}
 $$ \|u\|_{L^2( \T^d)} (\widetilde{T}) \leq 2^{N} \|u_0\|_{L^2(\T^d)}\leq 2^{ 1+ 2 \int_0^{t} C(s) ds} \|u_0\|_{L^2( \T^d)}.$$ The uniqueness claim in Proposition~\ref{propo.NLS} follows now from standard methods.
 \end{proof} 
\end{document}